\documentclass[11pt]{amsart}

\usepackage[margin=2cm]{geometry}

\usepackage{latexsym, amssymb, amsfonts, amsmath, amsthm, graphics, graphicx, bbm, stmaryrd,xcolor}
\usepackage[pdftex,colorlinks,linkcolor=blue,menucolor=red,backref=false,bookmarks=true,citecolor=olive]{hyperref}

\usepackage{palatino}

\newtheorem{theorem}{Theorem}[section]
\newtheorem{lemma}[theorem]{Lemma}
\newtheorem{corollary}[theorem]{Corollary}

\theoremstyle{definition}

\theoremstyle{remark}
\newtheorem{remark}[theorem]{Remark}

\newcommand{\field}[1]{\mathbb{#1}}
\newcommand{\R}{\field{R}}

\newcommand{\E}{\field{E}}

\newcommand{\tr}{\mathrm{tr}}

\newcommand{\mat}[2][rrrrrrrrrrrrrrrrrrrrrrrrrrrrrrrrrrrrrrrrrrrrrrrrrrr]{\left[ \begin{array}{#1} #2 \\ \end{array}\right]}

\allowdisplaybreaks

\numberwithin{equation}{section}
\begin{document}

\title[Couplings of Brownian Motions of deterministic distance]{Couplings of Brownian Motions of deterministic distance in model spaces of constant curvature}


\author{Mihai N. Pascu}
\address{"Transilvania" University of Bra\c{s}ov, Faculty of Mathematics and Computer Science, Str. Iuliu Maniu 50, Bra\c{s}ov 500091, ROMANIA.}
\email{mihai.pascu@unitbv.ro}
\thanks{The first author kindly acknowledges the support by a grant of the Romanian National Authority for Scientific Research, CNCS - UEFISCDI, project number PNII-ID-PCCE-2011-2-0015, and the second author the support by the Marie Curie Action Grant PIRG.GA.2009.249200 and Simons Collaboration Grant no.  318929. }

\author{Ionel Popescu}
\address{School of Mathematics, Georgia Institute of Technology, 686 Cherry Street, Atlanta, GA 30332, USA and ``Simion Stoilow'' Institute of Mathematics   of Romanian Academy, 21 Calea Grivi\c{t}ei, Bucharest, ROMANIA}
\email{ipopescu@math.gatech.edu,  ionel.popescu@imar.ro}

\subjclass[2010]{Primary 60J65. Secondary 60G99, 58J65.}

\keywords{Couplings, Brownian motion} 

\date{}

\dedicatory{}

\commby{Mark M. Meerschaert}

\begin{abstract}
We consider the model space $\field{M}^{n}_{K}$ of constant curvature $K$ and dimension $n\ge 1$  (Euclidean space for $K=0$, sphere for $K>0$ and hyperbolic space for $K<0$), and we show that given a function $\rho:[0,\infty)\to[0,\infty)$, $\rho(0)=d(x,y)$ there exists a co-adapted coupling $(X(t),Y(t))$ of Brownian motions on $\field{M}^{n}_{K}$ such that $\rho(t)=d(X(t),Y(t))$ for every $t\ge 0$ if and only if $\rho$ is continuous and satisfies for almost every $t\ge0$ the differential inequality
\[
-(n-1)\sqrt{K}\tan\left(\tfrac{\sqrt{K}\rho(t)}{2}\right)\le \rho'(t)\le -(n-1)\sqrt{K}\tan\left(\tfrac{\sqrt{K}\rho(t)}{2}\right)+\tfrac{2(n-1)\sqrt{K}}{\sin(\sqrt{K}\rho(t))}.
\]
In other words, we characterize all co-adapted couplings of Brownian motions on the model space $\field{M}^{n}_{K}$ for which the distance between the processes is deterministic. In addition, the construction of the coupling is explicit for every choice of $\rho$ satisfying the above hypotheses.
\end{abstract}

\maketitle


\section{Introduction }

The initial motivation for writing the present article was to investigate a stochastic version of the celebrated ``\emph{Lion and Man}'' problem of Rad\'{o} (\cite{Littlewood}), which asks for winning strategies in a game in which a Man is chased by a Lion in a circular arena (to make the game interesting, it is assumed that both trajectories are smooth and of unit speed). Substituting the circular arena with a model space, namely a manifold of constant curvature, and the smooth trajectories of the Man and the Lion by Brownian trajectories on this space, the problem becomes a question about the existence of Brownian couplings for which the distance between the two processes is bounded away from zero (of particular interest being the case when it is constant), respectively when it approaches (or becomes) zero.

In this paper the model spaces are classified by the curvature constant $K$.  In the case of $K=0$ this is the Euclidean space, in the case $K>0$ this is the sphere of radius $1/\sqrt{K}$ while in the case of $K<0$ is the hyperbolic space with a certain metric.

Regarding couplings, there are at least two cases of interest. The first is the case when one wishes to have a fast coupling time (the Lion's  strategy), and the second is the case when one desires a slow (infinite) coupling time (the Man's strategy; think also of the case of a Brownian target being chased by a Brownian missile). There are various notions of fast/slow couplings in the literature: see, e.g., \cite{Chen} for ``optimal couplings'', \cite{Burdzy-Kendall2} for ``efficient coupling'', or \cite{Hsu-Sturm} for ``maximal couplings''. In the case of Euclidean Brownian motions, the extremes of the coupling time are achieved by the \emph{mirror coupling}, respectively by the \emph{synchronous coupling}.

The \emph{mirror coupling} was first introduced in the Euclidean case by Lindvall and Rogers \cite{Lin-Rog}, and then extended to processes defined on manifolds by Cranston \cite{CranstonJFA} and Kendall \cite{Kendall5}, the so-called \emph{Cranston-Kendall mirror coupling}. For a recent extension of the mirror coupling to the case when the two (reflecting) Brownian motions live in different domains, see \cite{Pascu1}. It turns out that this coupling is a very useful and versatile construction when it comes to various geometric and analytic properties on manifolds. For instance, it was shown in \cite{Kendall5}, for the case of manifolds with Ricci curvature bounded uniformly from below by a positive constant, that mirror coupled Brownian motions always meet in finite time. Informally, the mirror coupling makes the Brownian motions move toward each other in the geodesic direction, and parallel to each other in the orthogonal direction.

Under \emph{synchronous coupling}, the Brownian motions move parallel to each other both in the geodesic direction and in the orthogonal direction. This coupling was used for example in \cite{Banuelos} to prove the validity of Hot Spots conjecture for obtuse triangles, or in \cite{MR2215664} to prove Harnack inequalities and heat kernel estimates on manifolds.  On a different note, parametrized families of Brownian motions were constructed in \cite{MR2790368} and \cite{Stroock2}.


In a different direction, the notion of \textit{shy coupling} of Brownian motions was introduced in \cite{Burdzy-Benjamini} and was further studied in \cite{Burdzy-Kendal} and \cite{Kendall}. It is a refinement of a slow coupling, for which with positive probability,  the distance between the two
processes stays positive for all times.   A stronger version of shyness ($\epsilon$-shyness, $\epsilon>0$) asserts that with positive probability, the distance between the processes is greater than $\epsilon$.  In this paper we use this latter version of shyness, in the stronger sense that the distance between the processes is greater than $\epsilon$ with probability $1$.

Thus, there are at least two cases of interest regarding Brownian couplings: on one hand is the case of couplings under which the Brownian motions get close to each other fast, and, on the other, is the case of couplings under which the particles eventually stay away from each other.

Our focus in this paper is the following problem.  Let $\rho$ be a continuous function of time with $\rho(0)$ being the initial distance between the Man and the Lion.   Does the Brownian Man (or Lion) have a strategy in this game, such that the distance between him and the Lion (or Man) at time $t$ is precisely $\rho(t)$?  Our goal is  to determine the class of functions $\rho$ for which this is possible together with a description of the coupling.

In the Euclidean case, given a function $\rho:[0,\infty)\to[0,\infty)$ we show that there exists a co-adapted coupling $\left(X(t),Y(t)\right)$ of Brownian motions such that $d(X(t),Y(t))=\rho (t)$ for all $t\ge0$ if and only if $\rho$ is continuous and satisfies a.e. the inequality
\[
0\le \rho'(t)\le \frac{n-1}{\rho(t)},\qquad t\ge0.
\]
In particular the only co-adapted coupling with deterministic non-increasing distance function is the translation coupling. This is the content of Theorem~\ref{t:01}.  As a consequence, the extreme growth of the distance function is achieved for the distance function $\rho(t)=\sqrt{\rho^{2}(0)+2(n-1)t}$. In terms of the stochastic version of the Lion and Man problem, this shows that there is no winning strategy for the Brownian Lion, while any co-adapted coupling with deterministic distance function is a winning strategy for the Brownian Man - the best choice is for the distance function which grows with square root in time.

The other case study is that of the unit sphere $\field{S}^{n}\subset \mathbb{R}^{n+1}$ ($n \ge 1$), and the analogous result is presented in Theorem~\ref{t:10}.  The content is that there exists a co-adapted coupling of Brownian motions $(X(t),Y(t))$ on $\field{S}^{n}$ such that $d(X(t),Y(t))=\rho(t)$ for all $t \ge 0$ if and only if $\rho:[0,\infty)\to[0,\infty)$ is continuous and satisfies a.e. the differential inequality
\[
-(n-1)\tan\left(\frac{\rho(t)}{2}\right)\le \rho'(t)\le -(n-1)\tan\left(\frac{\rho(t)}{2}\right)+\frac{2(n-1)}{\sin(\rho(t))}, \qquad t\ge0.
\]
An important instance of this result is the case when $\rho$ is a constant function, which gives the existence of a fixed-distance coupling of spherical Brownian motions. The two other interesting particular cases correspond to the cases of equality in the above inequality, being explicitly given by $\rho(t)=2\arcsin\left(e^{-(n-1)t/2}\sin(\rho(0)/2)\right)$, and $\rho(t)=2\arccos\left(e^{-(n-1)t/2}\cos(\rho(0)/2)\right)$. Both couplings are particular cases of shy couplings: in the former case the two Brownian motions approach each other exponentially fast, but do not couple, and in the latter case the two Brownian motions repel each other at an exponential rate to the maximum distance allowed on the sphere (the processes become antipodal in the limit).  In terms of the stochastic version of the Lion and Man problem, this shows that the Brownian Lion has a strategy which brings him within $\varepsilon>0$ from the Brownian Man, exponentially fast in time, and the Brownian Man has a strategy of safety, which increases his distance from the Brownian Lion to the maximum distance allowable on the sphere. Interestingly, the Brownian Man also has a strategy which keeps the Brownian Lion at fixed distance for all times, particularly frustrating from the point of view of the Brownian Lion.

The remaining case of a constant curvature manifold is that of the $n$-dimensional hyperbolic space $\field{H}^{n}$ ($n\ge1$).  In this case we show that there exists a coupling $(X(t),Y(t))$ of Brownian motions on $\field{H}^n$ with $d(X(t),Y(t))=\rho(t)$ if and only if $\rho:[0,\infty)\to[0,\infty)$ is a continuous function which satisfies almost everywhere the inequality
\[
(n-1)\tanh\left(\frac{\rho(t)}{2}\right)\le \rho'(t)\le (n-1)\tanh\left(\frac{\rho(t)}{2}\right)+\frac{2(n-1)}{\sinh(\rho(t))}, \qquad t\ge0.
\]
From the point of view of the Brownian Lion and the Brownian Man, in this scenario the Lion has no winning strategy, while the Man is always sure to get away from the Lion.  The interesting fact is that in this case the distance cannot grow exponentially fast, and in fact it has linear growth in time. To see this, from the previous inequality we obtain
\[
2\mathrm{arcsinh}(e^{(n-1)t/2}\sinh(\rho(0)/2))\le\rho(t)\le 2\mathrm{arccosh}(e^{(n-1)t/2}\cosh(\rho(0)/2)),
\]
which shows that $\rho(t)/t$ converges to $n-1$ as $t\to\infty$, and therefore the distance function $\rho(t)$ grows linearly in time for large times.

Denoting by $\field{M}^{n}_{K}$ the model space of the $n$-dimensional manifold of constant curvature $K$, i.e. the Euclidean space for $K=0$, the sphere of radius $1/\sqrt{K}$ for $K>0$, and the hyperbolic space in the case of $K<0$ (see Section \ref{s:p} for the details), we obtain the following unified form of our results given in the abstract.  This is the main result of this paper, and we include it here in a formal way.

\begin{theorem}\label{t:all}  For arbitrary distinct points $x,y\in \field{M}^{n}_{K}$ ($n\ge1$) and a given non-negative function $\rho:[0,\infty)\to[0,\infty)$ with $\rho(0)=d(x,y)$, there exists a co-adapted coupling of Brownian motions  $(X(t),Y(t))$ on $\field{M}^{n}_{K}$ starting at $(x,y)$ with deterministic distance function $\rho(t)=d(X(t),Y(t))$ if and only if $\rho$ is continuous on $[0,\infty)$ and satisfies for almost every $t\ge0$ the differential inequality
\begin{equation}\label{eg:1}
-(n-1)\sqrt{K}\tan\left(\tfrac{\sqrt{K}\rho(t)}{2}\right)\le \rho'(t)\le -(n-1)\sqrt{K}\tan\left(\tfrac{\sqrt{K}\rho(t)}{2}\right)+\tfrac{2(n-1)\sqrt{K}}{\sin(\sqrt{K}\rho(t))}.
\end{equation}

In particular, for dimension $n=1$, the only co-adapted coupling $(X(t),Y(t))$ of Brownian motions on $\field{M}^1_K$ with deterministic distance is given by
\[
Y(t)= \begin{cases}
e^{i\theta}X(t) & \text{if }K>0 \\
X(t)+\theta, & \text{if }K=0 \\
\theta X(t), & \text{if }K<0 \\
\end{cases}
\]
for some $\theta\in\R$, with $\theta>0$ if $K<0$.
\end{theorem}

Via a simple scaling argument, the proof of the theorem follows from the three cases discussed above: the case of the Euclidean space (Theorem \ref{t:01}), the case of the unit sphere $\field{S}^n$ (Theorem \ref{t:10}), and the case of the hyperbolic space $\field{H}^n$ (Theorem \ref{t:20}).

The outline of the paper is the following.  In Section~\ref{s:p} we introduce the basic notations and results needed in the sequel. For further use in the analy\-sis of co-adapted couplings of spherical Brownian motions, in Lemma \ref{p:1} we derive a characterization of all such couplings, similar to the one obtained in \cite{Emery} or \cite{Kendall} in the case of Euclidean Brownian motions, and intimately related to Stroock's representation of spherical Brownian motion.  Section~\ref{s:1} contains the analysis of Brownian couplings on $\R^{n}$ (Theorem~\ref{t:01}), and in Section~\ref{s:2} we present the analogous result for spherical Brownian motions (Theorem~\ref{t:10}).  In Section~\ref{s:3} we analyze the hyperbolic space case.   The paper concludes with Section~\ref{s:7}, which contains several remarks and corollaries of the main theorems, regarding the existence of fixed-distance, distance-increasing and distance-decreasing couplings, and an interpretation of the main results in terms of the stochastic Lion and Man problem.

We point out that from the geometric point of view the construction of the couplings is an extrinsic one, relying on Stroock's representation of spheri\-cal Brownian motion in terms of a Brownian motion in the ambient Euclidean space and the realization of the hyperbolic space as the half space in Euclidean space.  The advantage of this approach is that we can describe explicitly all co-adapted couplings, the downsize being that the construction applies only to model spaces. In another paper (\cite{pop}) we investigate, using an intrinsic approach, partial extensions of the results obtained in the present paper to the case of smooth manifolds without boundary.

\section{Preliminaries}\label{s:p}

We identify the vectors in $\mathbb{R}^{n}$ with the corresponding $n\times 1$ column matrices, and for a vector $x\in \mathbb{R}^{d}$ we denote by $x'$ the transpose of $x$. The dot product of two vectors $x,y\in \mathbb{R}^{n}$ will be written in terms of matrix multiplication as $x\cdot y=x'y$, and we will denote by $ | x | =\sqrt{x'x}$  the Euclidean length of the vector $x\in \mathbb{R}^{n}$.  We will use $\{e_{i}\}_{i=1,\dots,n}$ to denote the standard basis of $\R^{n}$.

We denote the \emph{$n$-dimensional unit sphere} in $\mathbb{R}^{n+1}$ by
\[
\field{S}^{n}=\left\{ x\in \mathbb{R}^{n+1}:| x| =1\right\}
\]
and by $d(x,y)$ the geodesic distance on it.  The relationship between the geodesic distance on $\field{S}^n$ and the Euclidean distance is given by
\begin{equation}\label{e1:ds}
d\left( x,y\right) =\arcsin \sqrt{1- \left( x'y\right)^2}=2\arcsin\left(\frac{| x-y|}{2}\right).
\end{equation}

There are various ways of describing the spherical Brownian motion on $\field{S}^{n} $, that is the Brownian motion living on the surface of the sphere $\field{S}^{n}$ (see for example \cite{Brillinger}). We will use Stroock's representation of spherical Brownian motion (\cite{Stroock}), as the solution $X(t)$ of the It\^{o}'s stochastic differential equation
\begin{equation}\label{e:sbm}
X(t)=X(0)+\int_0^t \left( I-X(s)X(s)'\right) dB(s)-\frac{n}{2}\int_0^t X(s)ds,  \qquad t\ge 0,
\end{equation}
where $B(t)$ is a $(n+1)$-dimensional Brownian motion.   The last term above may be thought as the pull needed in order to keep $X(t)$ on the surface of the sphere $\field{S}^{n}$.  In terms of Stratonovich integrals (see for instance \cite[Chapter 2]{Elton}), the above can be written equivalently as
\begin{equation}\label{e:sf:s}
X(t)=X(0)+\int_0^t \left(I-X(s)X(s)'\right)\circ dB(s),  \qquad t\ge 0,
\end{equation}
where ``$\circ$" denotes the Stratonovich integration.  Note that the operator $I-X(t)X(t)'$ is simply the projection from $\R^{n+1}$ onto the tangent space to the sphere at $X(t)$.  Equation \eqref{e:sbm} represents extrinsic formulation of the spherical Brownian motion on $\field{S}^{n}$, seen as a submanifold of $\R^{n+1}$.


The Laplacian $\Delta_{S}$ on the sphere can be computed as
\[
(\Delta_{S} f)(x)=(\Delta_{E} \hat{f})(x),\qquad x\in \field{S}^{n}
\]
where $\hat{f}(x)=f\left({x}/{|x|}\right)$, and $\Delta_{E}$ is the Euclidean Laplacian.   Using \eqref{e:sbm} it is not hard to see that $\hat{f}(X(t))-\frac{1}{2}\int_{0}^{t}\Delta_{E}\hat{f}(X(s))$ is a martingale for any smooth function $f$ on $\field{S}^{n}$, which proves that $X(t)$ is indeed the spherical Brownian motion.

We denote by $\field{S}_r^n=\{x\in\field{R}^{n+1}: |x|=r\}$ the $n$-dimensional sphere of radius $r>0$.  It is well known that the curvature of $\field{S}^{n}_{r}$ is constant equal to $1/r^{2}$, and for this reason $\field{S}^{n}_{r}$ is called the \emph{model space of constant curvature} $1/r^{2}$.   The Brownian motion $X^r(t)$ on $\field{S}^{n}_{r}$ is given by a simple rescaling of the Brownian motion $X(t)$ on the unit sphere $\field{S}^n$, and it can be described as $X^{r}(t)=rX(t/r^{2})$, with $X(t)$ solving \eqref{e:sbm}.

There are several models for the manifold of constant curvature $-1$, from which we choose here the half-space model.  Our model for the \emph{$n$-dimensional hyperbolic space} is thus given by
\[
\field{H}^{n}=\{(x_{1},\dots, x_{n})\in\R^{n}:x_{1}>0\}.
\]

For a point $x=(x_{1},x_{2},\ldots,x_{n})\in\field{H}^{n}$ we will often denote $\tilde{x}=(0,x_{2},\ldots,x_{n})$. The metric on $\field{H}^n$ is given by
\begin{equation}\label{e:hm}
ds^{2}=\frac{dx_{1}^{2}+dx_{2}^{2}+\dots +dx_{n}^{2}}{x_{1}^{2}},
\end{equation}
and the corresponding distance  is given by
\begin{equation}\label{e1:dh}
d(x,y)=\mathrm{arccosh}\left(\frac{| \tilde{x}-\tilde{y} |^{2}+x_{1}^{2}+y_{1}^{2}}{2x_{1}y_{1}}\right)=\mathrm{arccosh}\left(1+\frac{| x-y |^{2}}{2x_{1}y_{1}}\right),
\end{equation}
where $\mathrm{arccosh}$ denotes the inverse of the hyperbolic cosine function.

The Laplace-Beltrami operator on $\field{H}^n$ (see for instance \cite{Davies}) for details) is defined as 
\begin{equation}\label{e1:lh}
\Delta_{H}=x_{1}^{2}\sum_{i=2}^{n}\frac{\partial^{2}}{\partial x_{i}^{2}}+x_{1}^{2}\frac{\partial^{2}}{\partial x_{1}^{2}}-(n-2)x_{1}\frac{\partial}{\partial x_{1}}.
\end{equation}
The Brownian motion on $\field{H}^{n}$ is the diffusion process on $\field{H}^n$ with generator given by $\Delta/2$ above.  A different way of constructing the process (see for example \cite{Matsumoto}) is as the solution of the stochastic differential equation
\begin{equation}\label{e:hbm}
dX(t)=X_{1}(t)dB(t)-\frac{n-2}{2} X_{1}(t) e_{1}dt,
\end{equation}
where $B(t)$ is a $n$-dimensional Euclidean Brownian motion.   It is straightforward to check that the process $X(t)$ defined in this way is indeed the Brownian motion on $\field{H}^{n}$ by simply using the It\^o's formula and a straightforward calculation.

The space $\field{H}^{n}$ defined above is the space of constant curvature $-1$.  In a similar way we can define the \emph{model space of constant curvature} $-1/r^{2}$ ($r>0$) by $\field{H}^{n}_{r}=\{ x\in\R^{n}:x_{1}>0 \}$, with metric given by $r^{2}ds^{2}$, where $ds^{2}$ is given by \eqref{e:hm}.  The Brownian motion on this space is  defined by $X^{r}(t)=rX(t/r^{2})$, where $X(t)$ is the solution of \eqref{e:hbm}.

In order to give a unified statement of our results, we define the model space of  dimension $n\ge1$ and  constant curvature $K\in\R$ by
\[
\field{M}^{n}_{K}=\begin{cases}
\field{S}^{n}_{1/\sqrt{K}}, & \text{if } K>0 \\
\R^{n}, & \text{if } K=0 \\
\field{H}^{n}_{-1/\sqrt{-K}}, & \text{if } K<0
\end{cases}
.
\]

Recall that in general by a \emph{coupling} we understand a pair of processes $%
\left( X(t),Y(t)\right) $ defined on the same probability space, which are separately Markov, that is%
\[
\begin{split}
P\left( X(s+t)\in A| X(s)=z,X(u):0\leq u\leq s\right)
&=P^{z}\left( X(t)\in A\right) \\
P\left(  Y(s+t)\in A| Y(s)=z,Y_{u}:0\leq u\leq s\right)
&=P^{z}\left( Y(t)\in A\right)
\end{split}
\]
for any $s,t\ge 0$ and any measurable set $A$ in the state space of the processes.

The notion of \emph{Markovian coupling} as used in \cite{Burdzy-Benjamini} requires that in addition to the above, the joint process $(X(t),Y(t))$ is Markov and
\begin{equation}\label{eq:adapt}
\begin{split}
P\left(  X(s+t)\in A| X(s)=z,X(u),Y_{u}:0\leq u\leq
s\right) &=P^{z}\left( X(t)\in A\right) \\
P\left(  Y(s+t)\in A| Y(s)=z,X(u),Y_{u}:0\leq u\leq
s\right) &=P^{z}\left( Y(t)\in A\right)
\end{split}
\end{equation}
for any $s,t\ge0$ and any measurable set $A$ in the state space of the processes. Markovian couplings are easily obtained for instance in the case when the process $(X(t),Y(t))$ is actually a diffusion.   This would be the ideal case, but one can still construct a Markovian coupling by patching together diffusion processes for a conveniently chosen set of stopping times.

The notion of \emph{co-adapted coupling} (introduced by Kendall, \cite{Kendall}) is the same as the above but without the Markov property of $(X(t),Y(t))$. By a result on co-adapted couplings (Lemma 6 in \cite{Kendall}), a co-adapted coupling $\left( X(t),Y(t)\right) $ of Brownian motions in $\R^{n}$ can be represented as%
\begin{equation}\label{e:coc}
Y(t)=Y(0)+\int_0^t J(s)dX(s)+\int_0^t K(s)dC(s), \qquad t\ge0,
\end{equation}
where $C$ is a $n$-dimensional Brownian motion independent of $X$ (po\-ssi\-bly on a larger filtration), and $J ,K\in \field{M}_{n\times n}$ are matrix-valued predictable random processes, satisfy\-ing the identity
\begin{equation}\label{e:JK}
J(t)J(t)'+K(t)K(t)' =I,  \qquad t \ge 0,
\end{equation}%
with $I$ denoting the $n\times n$ identity matrix.

The formulae \eqref{e:coc} and \eqref{e:JK} provide an explicit representation of co-adapted couplings of $n$-dimensional Brownian motions, which will be used in order to characterize the couplings of deterministic distance in Euclidean spaces.  In order to derive the equivalent characterization for couplings of spherical Brownian motions, we need an explicit representation of co-adapted couplings on $\field{S}^{n}$. The following result gives such a characterization, intimately related to Stroock's representation \eqref{e:sbm} of the spherical Brownian motion.

\begin{lemma}\label{p:1}
Let $(X(t),Y(t))$ be a coupling of Brownian motions on $\field{S}^{n}$ ($n\ge1$).  The coupling $(X(t),Y(t))$ is co-adapted if and only if there exists a co-adapted coupling $(B(t),W(t))$ of $(n+1)$-dimensional Brownian motion, such that
\begin{equation}\label{e:c:s}
\left \{\begin{split}
X(t)&=X(0)+\int_0^t (I-X(s)X(s)')dB(s)-\frac{n}{2}\int_0^tX(s)ds \\
Y(t)&=Y(0)+\int_0^t (I-Y(s)Y(s)')dW(s)-\frac{n}{2}\int_0^t Y(s)ds
\end{split}, \qquad t\ge 0.
\right.
\end{equation}
\end{lemma}

\begin{proof}

Although the above is a statement about Brownian motions which is defined extrinsically in the language of differential geometry, in the proof we are going to use the intrinsic construction of the Brownian motion in terms of orthonormal frame bundle, as it is discussed for example in \cite{Elton} and \cite{Stroock2}.

The converse implication is easier to prove, and we begin with it. Let $\mathcal{F}=\left(\mathcal{F}_{t}\right)_{t\ge0}$ be the filtration generated by $B$ and $W$.  The hypothesis that the coupling $(B(t),W(t))$ is co-adapted shows that both $B(t)$ and $W(t)$ are Brownian motions with respect to the filtration $\mathcal{F}$.  Being solutions of \eqref{e:c:s}, $X(t)$ and $Y(t)$ are also adapted with respect to the filtration $\mathcal{F}$.  The conditions in \eqref{eq:adapt} follows now easily from the fact that both $X(t)$ and $Y(t)$ are Markov processes with respect to the filtration $\mathcal{F}$, hence they are also Markov processes with respect to the smaller filtration generated by $(X,Y)$.

To prove the direct implication, assume that $(X(t),Y(t))$ is a co-adapted coupling of Brownian motions on $\field{S}^n$ and let $\mathcal{F}=\left(\mathcal{F}_t\right)_{t\ge0}$ be the filtration generated by $(X,Y)$.  Condition \eqref{eq:adapt} shows that both $X$ and $Y$ are Markov processes with respect to the filtration $\mathcal{F}$, and since they are also diffusions on $\field{S}^{n}$, with one half the Laplacian on $\field{S}^{n}$ as generator, it follows that $X$ and $Y$ are semimartingales with respect to the filtration $\mathcal{F}$. To see this, considering a smooth and compactly supported function $f:\field{S}^{n}\rightarrow \R$ and denoting by $P_{t}$ the semigroup generated by $\frac{1}{2}\Delta_S$  on $\field{S}^{n}$, we have
\begin{align*}
&\E\left(\left.f(X(t))-\frac{1}{2}\int_{0}^{t}\Delta_S f(X(u))du\right|\mathcal{F}_{s}\right)=(P_{t-s}f)(X(s))-\frac{1}{2}\int_{0}^{t}\E\left(\Delta_S f(X(u))|\mathcal{F}_{s}\right)du\\
&=(P_{t-s}f)(X(s))-\frac{1}{2}\int_{0}^{s}\Delta_S f(X(u))du-\frac{1}{2}\int_{s}^{t} P_{u-s}\Delta_S f(X(u))du \\
&=(P_{t-s}f)(X(s))-\frac{1}{2}\int_{0}^{s}\Delta_S f(X(u))du-\int_{s}^{t} \frac{d}{du} P_{u-s} f(X(u))du \\
&=(P_{t-s}f)(X(s))-\frac{1}{2}\int_{0}^{s}\Delta_S f(X(u))du-P_{t-s} f(X(s))+f(X(s))\\
&=f(X(s))-\frac{1}{2}\int_{0}^{s} \Delta_S f(X(u))du
\end{align*}
for all $t\ge s\ge 0$, which shows that $f(X(t))-\frac{1}{2}\int_{0}^{t}\Delta_S f(X(u))du$ is a martingale with respect to the filtration $\mathcal{F}$.  In turn this implies that $X$ is a semimartingale with respect to the filtration $\mathcal{F}$, with a similar proof for $Y$.

The second step in proving the existence of the Brownian motions $B(t)$ and $W(t)$ satisfying \eqref{e:c:s} is to use orthonormal frame bundles in order  to construct their corresponding anti-development motions $\tilde{B}(t)$ and $\tilde{W}(t)$. The detailed construction  is clearly presented in \cite[Section 2.3]{Elton}, and we are not going to insist on it here.  Note that by construction, $\tilde{B}$ and $\tilde{W}$ are $\mathcal{F}_{t}$-adapted Brownian motions.

Given the starting points $x=X(0)$ and $y=Y(0)$ of the coupling, we can simply identify the tangent spaces $T_{x}\field{S}^{n}$ and $T_{y}\field{S}^{n}$ at these points with $\R^{n}$, and think of the Brownian motions $\tilde{B}$ and $\tilde{W}$ as moving in these spaces.  To complete the construction, we consider $1$-dimensional Brownian motions $\bar{B}(t)$ and $\bar{W}(t)$, independent of each other and also independent of $\mathcal{F}_{t}$, which take values on the line determined by $X(t)$ and $Y(t)$.   Let $\alpha_{t}$ be the parallel transport along the Brownian motion $X$ from $x$ to $X(t)$, and $\beta_{t}$ be the parallel transport from $y$ to $Y(t)$ along the Brownian motion $Y$.  It is easy to extend the operators $\alpha_{t}$ and $\beta_{t}$ to orthogonal matrices on $\R^{n+1}$, by simply defining them to send the vector $x$ into $X(t)$, respectively $y$ into $Y(t)$ (note that both $X(t)$ and $Y(t)$ are in the space $\field{S}^{n}$).

Consider the decompositions $\R^{n+1}=T_{x}\field{S}^{n}\times \R x$, $\R^{n+1}=T_{y}\field{S}^{n}\times \R y$ (with $\R x=\{ax: a\in\R\}$), and define
\[
dB(t)=\alpha_{t}d(\tilde{B}(t),\bar{B}(t))\quad\text{ and }\quad dW(t)=\beta_{t}d(\tilde{W}_{t},\bar{W}_{t}).
\]

Setting $\mathcal{G}_{t}$ to be be the $\sigma$-algebra generated by $\mathcal{F}_{t}$ and $\{\bar{B}_{s}, \bar{W_s}:0\le s\le t\}$, it is a simple matter to check that $B(t)$ and $W(t)$ are $\mathcal{G}_{t}$-adapted $(n+1)$-dimensional Brownian motions.

By construction (see \cite[Theorem 2.3.4 and Lemma 2.3.2]{Elton}) we have
\[
dX(t)=\left(I-X(t)X(t)'\right)\circ (\alpha_{t}d\tilde{B}(t))
\]
and since $\left(I-X(t)X(t)'\right)(\alpha_{t}d\tilde{B}(t))=\left(I-X(t)X(t)'\right)(\alpha_{t}dB(t))$, we obtain 
\[
dX(t)=\left(I-X(t)X(t)'\right)\circ dB(t),
\]
and similarly
\[
dY(t)=\left(I-Y(t)Y(t)'\right)\circ dW(t).
\]

The fact that $(B(t),W(t))$ form a co-adapted coupling follows now easily, concluding the proof. \qedhere

\end{proof}

We close this section with the analogous result given in the previous lemma, for the case of hyperbolic Brownian motion.

\begin{lemma}\label{p:2h}
Let $(X(t),Y(t))$ be a coupling of Brownian motions on the hyperbolic space $\field{H}^{n}$ ($n\ge1$).  The coupling $(X(t),Y(t))$ is co-adapted if and only if there exists a co-adapted coupling $(B(t),W(t))$ of $n$-dimensional Brownian motions, such that
\begin{equation}\label{e:c:sh}
\left \{\begin{split}
X(t)&=X(0)+\int_0^t X_1(s)dB(s)-\frac{n-2}{2}\int_0^t X_1(s) e_{1} ds \\
Y(t)&=Y(0)+\int_0^t Y_1(s) dW(s)-\frac{n-2}{2}\int_0^t Y_1(s) e_{1}ds
\end{split}
, \qquad t\ge 0.
\right.
\end{equation}

\end{lemma}

\begin{proof}  The direct proof is rather simple.  Let $\mathcal{F}_{t}$ be the sigma algebra generated by the Brownian motions $B(s)$ and $W(s)$ for $s\in[0,t]$.  Since the processes $X(t)$ and $Y(t)$ are defined by the stochastic differential equations \eqref{e:c:sh}, we see that the sigma algebra generated by $X_{s}$ and $Y_{s}$, $s\in[0,t]$ is included in $\mathcal{F}_{t}$.  Thus $X(t)$ and $Y(t)$ are Markov processes with respect to the filtration $\mathcal{F}_{t}$, which proves the result.

For the converse implication, consider the filtration $\mathcal{F}=\left(\mathcal{F}_{t}\right)_{t\ge0}$ generated by $(X,Y)$, and note that both $X$ and $Y$ are Markov processes with respect to the filtration $\mathcal{F}$ (by the co-adapted hypothesis), and therefore they are semimartingales with respect to $\mathcal{F}$ (the proof is similar to that in the previous lemma).

Defining
\[
B(t)=\int_{0}^{t}\frac{1}{X_{1}(t)}dX(t)+\frac{(n-2)t}{2}e_{1}\text{ and } W(t)=\int_{0}^{t}\frac{1}{Y_{1}(t)}dY(t)+\frac{(n-2)t}{2}e_{1},
\]
it is not difficult to check that $B$ and $W$ are Brownian motions with respect to the filtration $\mathcal{F}$.  For instance, we can use  Levy's criterion and show that $B(t)$ is an $\mathcal{F}_{t}$-martingale with quadratic variation given by $t I$, where $I$ is the $n\times n$ identity matrix.  The martingale characterization \eqref{e:hbm} of the hyperbolic Brownian motion $X$ shows that the $X_{1}(t)-\frac{n-2}{2}\int_{0}^{t}X_{1}(s)ds$ and $X_2(t),\ldots,X_n(t)$ are $\mathcal{F}_t$-martingales, from which we deduce that $B(t)$ is also a $\mathcal{F}_t$-martingale. Using again \eqref{e:hbm} and the definition of $B$ above, we conclude that $$\langle B_i,B_j\rangle_t=\int_0^t \frac{1}{X_1^2(s)}d\langle X_i,X_j\rangle _s=\int_0^t \frac{1}{X_1^2(s)}\delta_{ij} X_1^2(s)ds=\delta_{ij}t, \quad 1\le i,j\le n.$$

A similar proof applies to $W$, and the claim follows. 
\end{proof}

\section{Brownian couplings of deterministic distance in $\R^{n}$}\label{s:1}

The main result of this section is the following characterization of all co-adapted couplings $\left( X(t),Y(t)\right) $ of $n$-dimensional Brownian motions, for which the distance $| X(t)-Y(t)|$ is a deterministic function of $t\geq 0 $.


\begin{theorem}\label{t:01}
For any distinct points $x,y\in\R^{n}$ ($n\ge 1$) and an arbitrary non-negative function $\rho:[0,\infty)\to[0,\infty)$ with $\rho(0)=|x-y|$, there exists a co-adapted coupling of Brownian motions  $(X(t),Y(t))$ in $\R^{n}$ starting at $(x,y)$ with deterministic distance function $|X(t)-Y(t)| =\rho(t)$ if and only if $\rho$ is continuous on $[0,\infty)$ and satisfies almost everywhere the differential inequality
\begin{equation}\label{iff condition for rho}
0\le \rho'(t)\le \frac{2(n-1)}{\rho(t)},\qquad t\ge0.
\end{equation}

In particular, the only co-adapted coupling of Brownian motions in $\R^{n}$ with (deterministic) non-increasing distance function is the translation coupling, and the only co-adapted coupling of Brownian motions in $\R$ with deterministic distance function is the translation coupling.

\end{theorem}

\begin{proof}
If $(X,Y)$ is an arbitrary co-adapted coupling of $n$-dimensional Brownian motions, $Y$ can be represented as in (\ref{e:coc}), where $C$ is a $n$-dimensional Brownian motion independent of $X$, and the matrices $J,K$ satisfy (\ref{e:JK}).

Setting $Z(t)=X(t)-Y(t)$ and using It\^{o}'s formula we obtain
\[
\begin{split}
d| Z(t)| ^{2}=&2Z(t)'dZ(t)+\sum_{i=1}^{n}d\langle Z^{i}\rangle _{t} \\
=&2Z(t)'\left( I-J(t)\right) dX(t)-2Z(t)'K(t)dC(t)+\sum_{i=1}^{n}d\langle Z^{i}\rangle _{t}.
\end{split}
\]

Using the independence of $X$ and $C$, and the relation (\ref{e:JK}) we obtain
\[
\begin{split}
\sum_{i=1}^{n}d\langle Z^{i}\rangle _{t} &=\mathrm{tr}\left( \left(
I-J(t)\right) \left( I-J(t)\right)' +K(t)K(t)'\right) dt =2\left( n-\mathrm{tr}\left( J(t)\right) \right) dt.
\end{split}
\]

From the last two equations we arrive at
\[
d| Z(t)| ^{2}=2Z(t)'\left( I-J(t)\right)
dX(t)-2Z(t)'K(t)dC(t)+2\left( n-\mathrm{tr}\left( J(t)\right)
\right) dt.
\]%

To prove the first claim of the theorem, note that if $(X,Y)$ is a co-adapted coupling of Brownian  motions with deterministic distance function $\rho(t)=|Z(t)|$, the martingale part of $|Z(t)|^2$ must be identically zero, and therefore (using the independence of the Brownian motions $X$ and $C$) we must have
\begin{equation}\label{e:cace}
J(t)' Z(t)
=Z(t),\, K(t)' Z(t)=0, \text{ and }  \rho^2(t)=\rho^2(0)+\int_0^t 2 (n-\mathrm{tr}\left( J(s)\right) )ds
\end{equation}%
for all $t\geq 0$.

From \eqref{e:JK} it follows that
$$0\leq x'J(t)'J(t)x\leq x'J(t)'J(t)x+x'K(t)'K(t)x=x'x,\qquad x\in \mathbb{R}^{n},$$
which shows that $|J(t)x|\le| x |$, and in particular $|e_{i}'J(t)e_{i}|\le 1$ for $i=1,\ldots,n$, and thus $-n\le \tr(J(t))\le n$. Combining the above with the last equation in \eqref{e:cace}  we conclude that $\rho$ is a non-decreasing function which satisfies
\begin{equation}\label{aux1}
0\le \rho(t) \rho^\prime(t)=n-\mathrm{tr}(J(t))
\end{equation}
for almost every $t\ge 0$.
In particular, this shows that $\rho(t)\ge\rho(0)>0$, so $Z(t)\ne 0$ for any $t\ge 0$, and therefore we can find an orthonormal basis $\{\xi_1(t),\ldots,\xi_{n}(t) \}$ of $\R^{n}$ such that $\xi_1(t)=\frac{Z(t)}{|Z(t)|}$. Using the first equation in \eqref{e:cace}, we obtain
\begin{equation*}
n-\mathrm{tr}\left(J(t)\right)=\mathrm{tr}\left(I-J(t)\right)=\sum_{i=1}^{n} \xi_i^\prime(t) \left(I-J(t)\right) \xi_i(t) = \sum_{i=2}^{n} \xi_i^\prime(t) \left(I-J(t)\right) \xi_i(t).
\end{equation*}

Since 
$
| \xi_i^\prime(t) J(t)|^2= \xi_i^\prime (t) J(t) J(t)^\prime \xi_i(t) \le \xi_i^\prime(t) \xi_i(t)=1,
$
we obtain $|\xi_i^\prime(t) J(t) \xi_i(t) |\le |\xi_i^\prime(t) J(t)| | \xi_i(t) | \le |\xi_i(t)| = 1$.  Combining with the last equation this shows that $n-\mathrm{tr} (J(t))\le 2(n-1)$, which together with (\ref{aux1}) concludes the proof of the direct implication.

In order to prove the converse implication, we have to show that if the distance function $\rho$ satisfies the equation (\ref{iff condition for rho}) for almost every $t\ge 0$, we can find the matrices $J(t)$ and $K(t)$ satisfying \eqref{e:JK} and \eqref{e:cace}.

The hypothesis (\ref{iff condition for rho}) shows that $|Z(t)|=\rho(t)$ is a non-decreasing function of $t\ge0$, and in particular $\rho(t)\ge\rho(0)>0$, which shows that  $Z(t)$ is a non-zero vector for any $t\ge 0$. Consider an orthonormal basis $\{\xi_1(t),\ldots,\xi_{n}(t) \}$ of $\R^{n}$ such that $\xi_1(t)=\frac{Z(t)}{|Z(t)|}$, and let $J(t)'$ be the matrix with eigenva\-lues $1,\lambda(t),\ldots,\lambda(t)$ and corresponding eigenvectors $\xi_1(t),\xi_2(t),\ldots,\xi_n(t)$, where $\lambda(t)\in \R$ is yet to be determined.

With the above choice for $J(t)$, we have $J(t)'Z(t)=|Z(t)|J(t)'\xi_1(t)=|Z(t)|\xi_1(t)=Z(t)$, so the first equation in (\ref{e:cace}) is satisfied. We also have $\mathrm{tr}(J(t))=1+(n-1)\lambda(t)$, so the last equation in \eqref{e:cace} reduces to $\rho^2(t)=\rho^2(0)+\int_0^t 2(n-1)\left(1-\lambda(t)\right) dt$, which is satisfied for the choice  $\lambda(t)=1-\frac{1}{n-1}\rho(t)\rho^\prime(t)$ (and $0$ for definiteness if $\rho^\prime(t)$ is undefined, or if it does not satisfy (\ref{iff condition for rho})). Note that we have to choose $\lambda(t)$ only when $n\ge2$, and in this case from the hypothesis (\ref{iff condition for rho}) we have $\lambda(t)\in [-1,1]$.

To complete the proof, we need to show that we can also choose the matrix $K(t)$ such that the the second relation in \eqref{e:cace} and the relation (\ref{e:JK}) are satisfied. Considering $K(t)=P_t D_{1,t} P_t'$, where $P_t$ is the matrix with column vectors $\xi_1(t),\xi_2(t),\ldots,\xi_{n}(t)$, and $D_{1,t}, D_{2,t}$ are the diagonal matrices with diagonal entries $0,\sqrt{1-\lambda^2(t)},\dots,\sqrt{1-\lambda^2(t)}$, respectively $1,\lambda(t),\ldots,\lambda(t)$, we have $K(t)'Z(t)=|Z(t)| P_t D_{1,t}' P_t' \xi_1(t)=0$ and
\[
J(t)J(t)'+K(t)K(t)' =P_t\left(D_{2,t}^2+D_{1,t}^2\right) P_t'=P_t P_t'= I,
\]
concluding the proof of the first part of the theorem.

The second part follows now from the characterization given in the first part: if the distance function $\rho$ in non-increasing, from (\ref{iff condition for rho}) it follows that $\rho'(t)=0$ for almost every $t\ge 0$, so $\mathrm{tr}(J(t))=n$ by (\ref{e:cace}). The inequality \eqref{aux1} shows that the equality $\mathrm{tr}(J(t))=n$ can hold iff $J(t)$ is the identity matrix, so $K(t)$ is the null matrix by (\ref{e:JK}), and therefore by (\ref{e:coc})  we obtain $Y(t)=Y(0)-X(0)+X(t)$ for almost (hence all) $t\ge0$, i.e. $(X,Y)$ is a translation coupling.

Finally, the characterization given in the first part also shows that for $n=1$ we have $\rho^\prime=0$ a.e., so $\rho(t)=\rho(0)$ for all $t\ge 0$, and therefore $(X,Y)$ is a translation coupling, concluding the proof.
\end{proof}

\section{Couplings of deterministic distance on $\field{S}^{n}$}\label{s:2}

In this section we study the co-adapted couplings of Brownian motions on the unit sphere $\field{S}^{n}\subset\R^{n+1}$.  The main result is the following.

\begin{theorem}\label{t:10}
For any points $x,y\in \field{S}^{n}$ ($n\ge1$) with $x\ne\pm y$ and an arbitrary non-negative function $\rho:[0,\infty)\to[0,\infty)$ with $\rho(0)=d(x,y)$, there exists a co-adapted coupling of Brownian motions  $(X,Y)$ on $\field{S}^{n}$ starting at $(x,y)$ with deterministic distance function $d(X(t),Y(t)) =\rho(t)$ if and only if $\rho$ is continuous on $[0,\infty)$ and satisfies almost everywhere the differential inequality
\begin{equation}\label{e:301}
-(n-1)\tan\left(\frac{\rho(t)}{2}\right)\le \rho'(t)\le -(n-1)\tan\left(\frac{\rho(t)}{2}\right)+\frac{2(n-1)}{\sin(\rho(t))}, \qquad t\ge0.
\end{equation}

In particular, the only co-adapted coupling of Brownian motions on $\field{S}^{1}$ with deterministic distance function is the rotation coupling, i.e. $Y(t)=e^{i\theta}X(t)$ for some $\theta\in\R$ and all $t\ge0$.
\end{theorem}

\begin{proof}

To simplify the notation, we will prove an equivalent statement involving the function
\[
\eta (t)=X'(t)Y(t)=\cos\left(\rho(t) \right)
\]
Note that $\rho$ is a deterministic function iff $\eta$ is so, and the  inequality \eqref{e:301} is equivalent to
\begin{equation}\label{e:301s}
-(n-1)(\eta+1)\le \eta'(t)\le -(n-1)(\eta-1)
\end{equation}

Before proceeding to the proof, we will first argue that without loss of generality we may assume that $X(t)\ne \pm Y(t)$ for all $t\ge0$.   Assume that the theorem has been proved for all set of times in $t\in[0,T)$, where $T=\inf\{t\ge0:\eta(t)=1 \text{ or } \eta(t)=-1\}$. The double inequality in \eqref{e:301s} can be written in the equivalent form
\begin{equation}\label{e:4s}
-(n-1)\eta(t)-(n-1)\le \eta(t)'\le -(n-1)\eta(t)+(n-1),\qquad \text{for a.e. } t\in[0,T),
\end{equation}
which combined with Gr\"{o}nwall's lemma (for which the almost everywhere differentiability of $\eta$ suffices) gives
\begin{equation*}
-1+(\eta(0)+1) e^{-(n-1)t}\le \eta(t)\le 1+(\eta(0)-1)e^{-(n-1)t},\qquad t\in[0,T).
\end{equation*}

If $T$ were finite, by the continuity of $\eta$ we would obtain
\[\{-1,1\}\ni\eta(T)\subset[-1+(\eta(0)+1) e^{-(n-1)T}, 1+(\eta(0)-1)e^{-(n-1)T}]\subset(-1,1),
 \]
a contradiction. We conclude that without loss of generality we may assume that $X(t)\ne \pm Y(t)$, for all $t\ge0$, or equivalent that $\eta(t)\in (-1,1)$ for all $t\ge 0$.


Assume now that $(X,Y)$ is an arbitrary co-adapted coupling of Brownian motions on the sphere $\field{S}^{n}$.  By Lemma~\ref{p:1} there exists a co-adapted coupling $(B,W)$ of $(n+1)$-dimensional Brownian motions such that
\begin{equation}
\begin{cases}
X(t)=X(0)+\int_0^t \left( I-X(s)X(s)'\right) dB(s)-\frac{n}{2}\int_0^t X(s)ds \\
Y(t)=Y(0)+\int_0^t \left( I-Y(s)Y(s)'\right) dW(s)-\frac{n}{2}\int_0^t Y(s)ds,
\end{cases}
\end{equation}
and by Lemma $6$ in \cite{Kendall} we also have
\[
W(t)=\int_0^t J(s)dB(s)+\int_0^t K(s)dC(s),
\]
where $C$ is a $(n+1)$-dimensional Brownian motion independent of $B$, and $J(t), K(t)$ are $(n+1)\times (n+1)$ matrices satisfying
\begin{equation}  \label{Eq 0}
J(t)J(t)'+K(t)K(t)'=I, \qquad t\ge0.
\end{equation}

Let $U(t)=I-X(t)X(t)'$ and $V(t)=I-Y(t)Y(t)'$ (note that $U(t)$ and $V(t)$ are symmetric matrices, with $U(t)^{2}=U(t)$ and $V(t)^{2}=V(t)$).  We have
\begin{equation}\label{e:1s}
\begin{split}
d\eta(t)&=X(t)'dY(t)+Y(t)'dX(t)+\sum_{k=1}^{n+1}d\langle X_{k}, Y_{k}\rangle_t \\
&=(X(t)'V(t)J(t)+Y(t)'U(t))dB(t)+X'(t)V(t)K(t)dC(t) \\
&\qquad-nX(t)'Y(t)dt+\tr(U(t)J(t)'V(t))dt \\
&=(X(t)'V(t)J(t)+Y(t)'U(t))dB(t)+X'(t)V(t)K(t)dC(t) \\
&\qquad-n\eta(t)dt+\tr(U(t)J(t)'V(t))dt.
\end{split}
\end{equation}
Enforcing that $\eta$ is a deterministic function requires the cancelation of the martingale part of it. Since the Brownian motions $B$ and $C$  are independent, this is equivalent to
\[
J(t)'V(t)X(t)=-U(t)Y(t) \text{ and } K(t)'V(t)X(t)=0.
\]
Using the definition of $\eta(t)$, $U(t)$, and $V(t)$, the above is also equivalent to
\begin{equation}\label{e:2s}
J(t)'(X(t)-\eta(t)Y(t))=\eta(t)X(t)-Y(t)\text{ and }K(t)'(X(t)-\eta(t)Y(t))=0.
\end{equation}

To compute the trace $\tr(U(t)J(t)'V(t))$ in \eqref{e:1s} we will use the following representation of the trace of a $(n+1)\times(n+1)$ matrix $A$: $\mathrm{tr}(A)=\sum_{i=1}^{n}a_{ii}$, where $Af_{i}=\sum_{j=1}^{n}a_{ij}f_{j}$ and $f_{1},\ldots,f_{n+1}$ is an arbitrary basis of $\R^{n+1}$.  Since $X(t)\ne \pm Y(t)$, the subspace $\mathcal{U}(t)$ spanned by $X(t)$ and $Y(t)$ is $2$-dimensional, and let $\mathcal{V}(t)\subset\R^{n+1}$ be the $(n-1)$-dimensional subspace orthogonal to it.  Using \eqref{e:2s}, we obtain
\[
\begin{cases}
U(t)J(t)'V(t)X(t)=\eta(t)X(t)-Y(t)\\
U(t)J(t)'V(t)Y(t)=0.
\end{cases}
\]

If $n=1$, $X(t),Y(t)$ is a basis of $\mathcal{U}(t)=\R^2$, and from the above representation of the trace we obtain $\tr(U(t)J(t)'V(t))=\eta(t)$, which together with \eqref{e:1s} proves the claim \eqref{e:301s} in this case. If $n\ge2$, considering an arbitrary orthonormal basis $f_{1}(t),\ldots,f_{n-1}(t)$ of $\mathcal{V}_{t}$ and using the above representation of the trace, we obtain

\[
\tr(U(t)J(t)'V(t))=\eta(t)+\sum_{j=1}^{n-1}f_j(t)'  U(t)J(t)'V(t) f_j(t),
\]
and combining with \eqref{e:1s}, we obtain
\begin{equation}\label{e:3s}
d(\eta(t))=-(n-1)\eta(t)dt+\sum_{j=1}^{n-1}f_{j}(t)'U(t)J(t)'V(t)f_{j}(t) dt.
\end{equation}

The above proves the claim \eqref{e:301s}, if we show that each term in the sum above lies between $-1$ and $1$. To show this, note that the condition \eqref{Eq 0} implies that $|J(t)'\xi|\le |\xi|$ for any vector $\xi\in\R^{n+1}$,  so the operator norm of $J(t)'$ is at most $1$.  It follows that the operator norm of $J(t)$ is also at most $1$, and since $U(t)$ and $V(t)$ are projection operators, it is trivial to see that $|f_{j}(t)'U(t)J(t)'V(t)f_{j}(t)|\le1$, for any $j=1,\ldots,n-1$, concluding the proof of the direct implication of the theorem.

To prove the converse, we have to show that if $\eta$ satisfies \eqref{e:301s}, we can find $(n+1)\times(n+1)$ matrices $J(t),K(t)$ which satisfy \eqref{Eq 0}, \eqref{e:2s} and \eqref{e:3s}.  To this aim consider the matrix $J(t)$ defined by
\[
\begin{cases}
J(t)'X(t)=-Y(t) \\
J(t)'Y(t)=-X(t)\\
J(t)'V=\gamma(t)V,\qquad \text{ for any }V\in\mathcal{V}(t)
\end{cases}
,
\]
where
\[
\gamma(t)=1-\frac{\eta'(t)+(n-1)\eta(t)}{n-1},
\]
which from the condition \eqref{e:301s} is guaranteed to belong to $[-1,1]$.  Consider also the matrix $K(t)$  defined by
\[
\begin{cases}
K(t)'X(t)=0 \\
K(t)'Y(t)=0\\
K(t)'V=\sqrt{1-\gamma^2(t)}V,\qquad \text{ for any }V\in\mathcal{V}(t)
\end{cases}
,
\]
and note that with the above choices we have $J(t)=J(t)'$, $K(t)=K(t)'$, $J(t)J(t)'+K(t)K(t)'=I$.  In addition it is easy to check \eqref{e:2s} and \eqref{e:3s}, hence concluding the proof.
\end{proof}

\section{Couplings of deterministic distance on the hyperbolic space}\label{s:3}

\begin{theorem}\label{t:20} For any distinct points $x,y\in \field{H}^{n}$ ($n\ge1$)  and an arbitrary non-negative function $\rho:[0,\infty)\to[0,\infty)$ with $\rho(0)=d(x,y)$, there exists a co-adapted coupling of Brownian motions  $(X,Y)$ on $\field{H}^{n}$ starting at $(x,y)$ with deterministic distance function $\rho(t)=d(X(t),Y(t))$ if and only if $\rho$ is continuous on $[0,\infty)$ and satisfies almost everywhere the differential inequality
\begin{equation}\label{eh:301}
(n-1)\tanh\left(\frac{\rho(t)}{2}\right)\le \rho'(t)\le (n-1)\tanh\left(\frac{\rho(t)}{2}\right)+\frac{2(n-1)}{\sinh(\rho(t))}, \qquad t\ge0.
\end{equation}

In particular, in dimension $n=1$, the only co-adapted Brownian coupling $(X,Y)$ of deterministic distance on $\field{H}^1$ is given by $Y(t)= \theta X(t)$, $t\ge0$, for some $\theta>0$.
\end{theorem}

\begin{proof}  Before proceeding with the proof, note that by Lemma \ref{p:2h}, the Brownian coupling $(X,Y)$ on $\field{H}^n$ is co-adapted iff there exists a co-adapted coupling $(B,W)$ of  Euclidean Brownian motions  such that
\[
\left\{
\begin{split}
X(t)&=X(0)+\int_0^t X_1(s)dB_{s}-\frac{n-2}{2}\int_0^t X_1(s)e_{1}ds \\
Y(t)&=Y(0)+\int_0^t Y_1(s)dW_{s}-\frac{n-2}{2}\int_0^t Y_1(s)e_{1}ds
\end{split}
\right.
, \qquad t\ge0,
\]
and $dW(t)=J(t)dB(t)+K(t)dC(t)$, with $B,C$ independent $n$-dimensional Euclidean Brownian motions and  $J(t),K(t)$ $n\times n$ matrices satisfying $J(t)J(t)'+K(t)K(t)'=I$.  Denoting $Z(t)=X(t)-Y(t)$, the statement of the theorem can be expressed equivalently in terms of the function
\begin{equation}\label{eta:h}
\eta(t)=\frac{|Z(t)|^{2}}{2X_{1}(t)Y_{1}(t)}.
\end{equation}
The equation \eqref{e1:dh} shows that $\cosh(d(X(t),Y(t)))=\eta(t)+1$, and the inequality \eqref{eh:301} is therefore equivalent to
\begin{equation}\label{eh:5}
\eta(t)(n-1) \le \eta'(t)\le \eta(t)(n-1)+2(n-1), \qquad t\ge0.
\end{equation}

In order to prove the direct implication, we assume that $(X,Y)$ is a co-adapted Brownian coupling on $\field{H}^n$ with deterministic distance $\rho$ and show that \eqref{eh:5} holds. We first note, just as in the spherical case, that without loss of generality we may assume that $X(t)\ne Y(t)$ for all $t\ge 0$: if \eqref{eh:5} is satisfied for $0\le t<T=\inf \{s\ge0: \eta(s)=0\}$, then
\[
\eta(0)e^{(n-1)t}\le \eta(t), \qquad t\in[0,T),
\]
and thus $\eta$ does not vanish on $[0,T)$.  If $T$ were finite, by the continuity of $\eta$ we would obtain $\eta(T)=0$, a contradiction. Consequently we may assume that $X(t)\ne Y(t)$ for all $t\ge0$.

Using the preliminary remarks, we obtain
\[
dZ(t)=(X_{1}(t)I-Y_{1}(t)J(t))dB(t)-Y_{1}(t)K(t)dC(t)-\frac{n-2}{2}Z_{1}(t)e_{1}dt,
\]
and since $\eta(t)$ is a function of bounded variation, from \eqref{eta:h} we obtain
\begin{equation}\label{eh:2}
d|Z(t)|^{2}=2X_{1}(t)Y_{1}(t)\eta'(t)dt+2\eta(t) d(X_{1}(t)Y_{1}(t)).
\end{equation}
Since
\[
\begin{split}
d(X_{1}(t)Y_{1}(t))&=Y_{1}(t)e_{1}'dX(t)+X_{1}(t)e_{1}'dY(t)+d\langle X_{1},Y_{1} \rangle_{t} \\
&=X_{1}(t)Y_{1}(t)e_{1}'\left(I+J(t)\right)dB(t)+X_{1}(t)Y_{1}(t)e_{1}'K(t)dC(t) \\
&\quad+X_{1}(t)Y_{1}(t)\left(e_{1}'J(t)e_{1}-(n-2)\right)dt,
\end{split}
\]
combining with the above we conclude
\[
\begin{split}
d|Z(t)|^{2}&= 2\eta(t) X_{1}(t)Y_{1}(t)e_{1}'\left(I+J(t)\right)dB(t)+ 2\eta(t) X_{1}(t)Y_{1}(t)e_{1}'K(t)dC(t) \\
&\quad +2 X_{1}(t)Y_{1}(t)\left(\eta(t) \left(e_{1}'J(t)e_{1}-(n-2)\right)+ \eta'(t)\right) dt,
\end{split}
\]

Alternately, we can determine the semimartingale decomposition of $|Z(t)|^2$ as follows:
\[
\begin{split}
d|Z(t)|^{2}&=2Z(t)'dZ(t)+ \sum_{i=1}^{n}d\langle Z_{i} \rangle_{t}\\
&=2M(t)dB(t)+2N(t)dC(t)-(n-2)Z_1^2(t) dt+ \sum_{i=1}^{n}d\langle Z_{i} \rangle_{t},
\end{split}
\]
where
\[
M(t)=X_{1}(t)Z(t)'-Y_{1}(t)Z(t)'J(t),\qquad N(t)=-Y_{1}(t)Z(t)'K(t),
\]
and
\[
\begin{split}
\sum_{i=1}^{n}d\langle Z_{i} \rangle_{t}&=\tr\left(\left(X_{1}(t)I-Y_{1}(t)J(t)\right)\left(X_{1}(t)I-Y_{1}(t)J(t)\right)'+Y_{1}^{2}(t)K(t)K(t)'\right)dt \\
&=\left( n X_{1}^{2}(t)+n Y_{1}^{2}(t)-2X_{1}(t)Y_{1}(t)\tr(J(t))\right)dt.
\end{split}
\]

Comparing the two semimartingale representations of $|Z|^2$ above, and using the independence of the Brownian motions $B$ and $C$, we conclude
\[
\left\{
\begin{split}
&M(t)=\eta(t)X_{1}(t)Y_{1}(t)e_{1}'(I+J(t))\\
&N(t)=\eta(t)X_{1}(t)Y_{1}(t)e_{1}'K(t)\\
&X_{1}^{2}(t)+Y_{1}^{2}(t)+X_{1}(t)Y_{1}(t)\left((1+\eta(t))\left(n-2-e_1'J(t)e_1\right)-\eta'(t)-\tilde{\tr}(J(t))\right)=0.
\end{split}
\right.
\]
where $\tilde{\tr}(J(t))=\sum_{i=2}^{n}e_{i}'J(t)e_{i}$. Using the definitions of $M(t)$, $N(t)$, and $\eta$ (also recall the notation $\tilde{x}=(0,x_2,\ldots,x_n)$), the above conditions are equivalent to
\begin{equation}\label{eh:9}
\begin{cases}
\begin{split}
J(t)'&\left(\left(|\tilde{Z}(t)|^{2}+ X_{1}^{2}(t)-Y_{1}^{2}(t)\right)e_{1}+2Y_{1}(t)\tilde{Z}(t)\right)\\
&\qquad\qquad=\left((-|\tilde{Z}(t)|^{2}+X_{1}^{2}(t)-Y_{1}^{2}(t))e_{1}+2X_{1}(t)\tilde{Z}(t)\right)
\end{split} \\ \\

K(t)'\left(\left(|\tilde{Z}(t)|^{2}+ X_{1}^{2}(t)-Y_{1}^{2}(t)\right)e_{1}+2Y_{1}(t)\tilde{Z}(t)\right)=0\\
\end{cases}
\end{equation}
and
\begin{equation}\label{eh:7}
\eta'(t)=(1+\eta(t))\left(n-2-e_{1}'J(t)e_{1}\right)-\tilde{\tr}(J(t))+\frac{X_{1}^{2}(t)+Y_{1}^{2}(t)}{X_{1}(t)Y_{1}(t)}.
\end{equation}

Consider first the case when $\tilde{Z}(t)=0$. Since $Z(t)\ne0$ we must have $X_1(t)\ne Y_1(t)$, and from the first equation in \eqref{eh:9} we obtain $J(t)'e_{1}=e_{1}$.  Using this, $1+\eta(t)=\frac{X_{1}^{2}(t)+Y_{1}^{2}(t)}{2X_{1}(t)Y_{1}(t)}$, and \eqref{eh:7} gives
\begin{equation}\label{eh:8}
\eta'(t)=\eta(t)(n-1)+(n-1)-\tilde{\tr}(J(t)).
\end{equation}
Since $J(t)J(t)'+K(t)K(t)'=I$, we obtain $-1\le e_{i}'J(t)'e_{i}\le 1$, so $-(n-1)\le\tilde{\tr}(J(t))\le n-1$, and thus we arrive at \eqref{eh:5}.

For the case of $\tilde{Z}(t)\ne 0$ we will use the following result.

\begin{lemma}\label{l:6}  Assume $A:\R^{n}\to \R^{n}$ is a linear map such  that $|Ax |\le |x|$ for any $x\in\R^{n}$.  Let $\xi_{1},\xi_{2}\in\R^{n}$ be orthonormal vectors so that
\begin{equation}\label{eh:3}
A(m\xi_{1}+l\xi_{2})=p\xi_{1}+q\xi_{2}
\end{equation}
for some $m,l,p,q\in\R$ with $m^2+l^2=p^2+q^2\ne0$. The following hold true:
\begin{enumerate}
\item  $(I_{n}-A'A)^{1/2}(m\xi_{1}+l\xi_{2})=0$.

\item We have for any $r,s\in\R$,
\[
\frac{rmp+slq-|rlq+smp|}{m^{2}+l^{2}}\le r \xi_{1}'A\xi_{1}+s \xi_{2}'A\xi_{2}\le \frac{rmp+slq+|rlq+smp|}{m^{2}+l^{2}}.
\]

\item Moreover, for each $r,s,\phi\in\R$ satisfying
\[
\frac{rmp+slq-|rlq+smp|}{m^{2}+l^{2}}\le \phi \le \frac{rmp+slq+|rlq+smp|}{m^{2}+l^{2}}
\]
we have $r \xi_{1}'A_{\phi}\xi_{1}+s \xi_{2}'A_{\phi}\xi_{2}=\phi$, where $A_{\phi}$ is the map defined by
\begin{equation}\label{eh:6}
\begin{cases}
A_{\phi}\xi_{1}=\displaystyle \frac{mp+dlq}{m^{2}+l^{2}}\xi_{1}+\frac{mq-dlp}{m^{2}+l^{2}}\xi_{2} \\
A_{\phi}\xi_{2}=\displaystyle\frac{lp-dmq}{m^{2}+l^{2}} \xi_{1}+\frac{lq+dmp}{m^{2}+l^{2}}\xi_{2}\\
A_{\phi}\xi=0,\, \text{ for }\xi\text{ orthogonal to }\xi_{1},\xi_{2},
\end{cases}
\end{equation}
with $d=\frac{\phi(m^{2}+l^{2})-(rmp+slq)}{rlq+smp}$.
\end{enumerate}
\end{lemma}

\begin{proof}

\begin{enumerate}
\item
Note that if $D$ is a symmetric, non-negative definite matrix for which $D\xi=0$, then $| D^{1/2}\xi|^{2}=\xi'D\xi=0$, and therefore $D^{1/2}\xi=0$.
Using this with $D=I-A'A$ (which is non-negative definite) and $\xi=m\xi_{1}+l\xi_{2}$, one can check by  direct calculation that $\xi'D\xi=m^{2}+l^{2}-p^{2}-q^{2}=0$ from which the first claim follows.

\item
It is clear that we have to focus on the action of the matrix $A$ on the space $\mathcal{U}$ generated by $\xi_{1}$ and $\xi_{2}$.  For convenience,  let $\mathcal{V}$ be the subspace orthogonal to $\mathcal{U}$, and consider the restriction of $\tilde{A}$ to the subspace $\mathcal{U}$, more precisely, $\tilde{A}\xi=\Pi A\xi$ for where $\Pi$ is the orthogonal projection onto $\mathcal{U}$.

Writing vectors in $\mathcal{U}$ in terms of the orthonormal basis $\xi_{1},\xi_{2}$, we can simply assume that $\R^{n}$ is replaced by $\R^{2}$, and the map $A$ is replaced by  $\tilde{A}$. It is clear that the relation \eqref{eh:3} continues to hold with $A$ replaced by $\tilde{A}$.  For simplicity we will also assume that  $m^{2}+l^{2}=p^{2}+q^{2}=1$.  Consider
\[
U=\mat{m &-l \\ l & m}\text{ and }V=\mat{p&-q\\q&p}.
\]
We obviously have
\[
U\mat{1\\0}=\mat{m\\l} \text{ and } V\mat{1\\0}=\mat{p \\ q},
\]
so $\tilde{A}Ue_{1}=Ve_{1}$.  Denoting $B=V'\tilde{A}U=\mat{a & b\\ c &d}$, then
\[
Be_{1}=e_{1}
\]
which means that $a=1$ and $c=0$.  On the other hand, we also have that $|Bx|\le |x|$ for any $x\in\R^{2}$, which implies that for $x=\mat{t\\s}$,
\[
(t+bs)^{2}+d^{2}s^{2}\le t^{2}+s^{2}
\]
for any real numbers $t,s$.  Thus, we must have $b=0$ and $-1\le f\le 1$.  Therefore,
\[
\tilde{A}=VBU' =\mat{mp+dlq & pl-dmq \\qm-dpl & lq+dmp}.
\]
Thus,
\[
re_{1}'\tilde{A}e_{1}+se_{2}'\tilde{A}e_{2}= rmp+slq+d(rlq+smp).
\]
Since $-1\le d\le 1$, the rest follows easily.

\item
This follows from a direct calculation and is left to the reader.
\end{enumerate}
\end{proof}

Using the above lemma with $A=J(t)'$, $\xi_{1}=e_{1}$, $\xi_{2}=\tilde{Z}/|\tilde{Z}|$, $m=|\tilde{Z}_{t}|^{2}+X_{1}^{2}(t)-Y_{1}^{2}(t)$, $l=2Y_{1}(t)|\tilde{Z}_{t}|$, $p=-|\tilde{Z}_{t}|^{2}+X_{1}^{2}(t)-Y_{1}^{2}(t)$, $q=2X_{1}(t)|\tilde{Z}_{t}|$, $r=1+\eta(t)$, and $s=1$, we have
\[
\frac{rmp+slq}{m^{2}+l^{2}}=\frac{X_{1}^{2}(t)+Y_{1}^{2}(t)}{X_{1}(t)Y_{1}(t)}-\eta(t)-1 \quad\text{ and } \quad\frac{rlq+smp}{m^{2}+l^{2}}=1,
\]
and therefore we obtain
\[
a-1\le (1+\eta(t))e_{1}'J(t)'e_{1}+\xi_{2}'J(t)'\xi_{2} \le a+1,
\]
where
\[
a=\frac{X_{1}^{2}(t)+Y_{1}^{2}(t)}{X_{1}(t)Y_{1}(t)}-\eta(t)-1.
\]

Completing $\{e_1,\xi_2\}$ to an orthonormal basis $\{e_1,\xi_2,\ldots,\xi_n\}$ of $\R^n$, and recalling that $-1\le \xi_i'J(t)'\xi_i\le 1$ for any unit vector $\xi_i$, we have
\begin{equation*}
a-n+1\le(1+\eta(t))e_i'J(t)'e_i+\tilde{\mathrm{tr}}(J(t))\le a+n-1,
\end{equation*}
and therefore using \eqref{eh:7} we obtain
\[
\eta(t)(n-1)+(n-2)-\tr(\tilde{J})\le  \eta'(t)\le \eta(t)(n-1)+n-\tr(\tilde{J}),
\]
concluding the proof of the direct implication.

Conversely, assuming \eqref{eh:5} holds, we have to show that we can choose matrices $J(t),K(t)$ such that $J(t)J(t)'+K(t)K(t)'=I$, \eqref{eh:9} and \eqref{eh:7} are satisfied.

In the case $n=1$, choosing $J(t)=-1$ and $K(t)=0$ proves the claim.

Consider now the case $n\ge2$. If $\tilde{Z}(t)=0$, the equation \eqref{eh:7} is equivalent to \eqref{eh:8}, and therefore the claim is satisfied for the choice
\[
J(t)e_{i}=
\begin{cases}
e_{i} & i=1 \\
\gamma(t)e_{i} & i\ne1
\end{cases}
\quad\text{ and }\quad
K(t)e_{i}=
\begin{cases}
0 & i=1 \\
\sqrt{1-\gamma^{2}(t)}e_{i} & i\ne1
\end{cases}
\]
where
\begin{equation}\label{eh:10}
\gamma(t)=1+\eta(t)-\frac{\eta'(t)}{n-1}\in[-1,1].
\end{equation}
For the case $\tilde{Z}(t)\ne0$, we will use Lemma \ref{l:6} as follows. Consider $\xi_{1}=e_{1}$, $\xi_{2}=\tilde{Z}(t)/|\tilde{Z}(t)|$, $m=|\tilde{Z}(t)|^{2}+X_{1}^{2}(t)-Y_{1}^{2}(t)$, $l=2Y_{1}(t)|\tilde{Z}(t)|$, $p=-|\tilde{Z}(t)|^{2}+X_{1}^{2}(t)-Y_{1}^{2}(t)$, $q=2X_{1}(t)|\tilde{Z}(t)|$, $r=1+\eta(t)$, and $s=1$, for which
\[
\frac{rmp+slq}{m^{2}+l^{2}}=\frac{X_{1}^{2}(t)+Y_{1}^{2}(t)}{X_{1}(t)Y_{1}(t)}-\eta(t)-1 \quad\text{ and } \quad\frac{rlq+smp}{m^{2}+l^{2}}=1.
\]

Using the notation of the lemma, define the $n\times n$ matrices $J$ and $K$ by
\[
J(t)\xi=
\begin{cases}
A_{\phi(t)}'\xi, & \xi\in \mathcal{U}(t) \\
\gamma(t)\xi, & \xi\in \mathcal{V}(t)
\end{cases}
\text{ and }
K(t)\xi=
\begin{cases}
(I-A_{\phi(t)}'A_{\phi(t)})^{1/2}\xi, & \xi\in \mathcal{U}(t) \\
\sqrt{1-\gamma^{2}(t)}\xi, & \xi\in\mathcal{V}(t)
\end{cases}
,
\]
where $\mathcal{U}(t) = \mathrm{span} \{ e_{1},\tilde{Z}(t) \}$, $\mathcal{V}(t)\subset\R^n$ is the subspace orthogonal to it, $\gamma(t)$ is given by \eqref{eh:10}, and $\phi(t)$ (whose value will be chosen below) satisfies
\begin{equation}\label{eh:11}
\frac{X_{1}^{2}(t)+Y_{1}^{2}(t)}{X_{1}(t)Y_{1}(t)}-\eta(t)-2\le \phi(t)\le\frac{X_{1}^{2}(t)+Y_{1}^{2}(t)}{X_{1}(t)Y_{1}(t)}-\eta(t).
\end{equation}

Note that $I-A_{\phi(t)}' A_{\phi(t)}$ is a symmetric non-negative definite matrix and $\gamma(t)\in[-1,1]$, so the matrix $K(t)$ is well defined (for definiteness, if $\eta'(t)$ is not defined, or if it does not satisfy \eqref{eh:5}, we consider $\gamma (t)=0$).

With the above choice for $J$ and $K$, the condition  $J(t)J(t)'+K(t)K(t)'=I$ is clearly satisfied for any $\phi(t)$ satisfying \eqref{eh:11}. The first equation in \eqref{eh:9}can be written equivalently as
\[
J(t)'\left(m\xi_1+l\xi_2\right)=A_{\phi(t)}\left(m\xi_1+l\xi_2\right)=p\xi_1+q\xi_2,
\]
and is satisfied by the definition \eqref{eh:6} of the matrix $A_{\phi(t)}$, for any $\phi(t)$ sa\-tis\-fying \eqref{eh:11}. The above relation (together with $|\gamma(t)|\le 1$) also shows that we can apply the first part of Lemma \ref{l:6} with $A=A_{\phi(t)}$, thus obtaining
\[
K'(t)\left( m\xi_1+l\xi_2\right)=(I-A_{\phi(t)}'A_{\phi(t)})^{1/2}\left( m\xi_1+l\xi_2\right)=0,
\]
so the second equation in \eqref{eh:9} is also satisfied, for any any $\phi(t)$ verifying \eqref{eh:11}.

Finally, we will choose $\phi(t)$ so that the equation \eqref{eh:7} is also satisfied. The last part of the lemma shows that
\[
(1+\eta(t))e_1'J(t)e_1+\frac{\tilde{Z}(t)'}{|\tilde{Z}(t)}J(t)\frac{\tilde{Z}(t)}{|\tilde{Z}(t)}=\phi(t),
\]
and completing $\{e_1,\frac{\tilde{Z}(t)'}{|\tilde{Z}(t)}\}$ to an orthonormal basis of $\R^n$, as in the previous part of the proof, we obtain $(1+\eta(t))e_1'J(t)e_1+\tilde{\mathrm{tr}}(J(t)=\phi(t)+(n-2)\gamma(t)$.  The condition \eqref{eh:7} is thus equivalent to
\[
\eta'(t)=(n-2)(1+\eta(t)-\gamma(t))- \phi(t)+\frac{X_{1}^{2}(t)+Y_{1}^{2}(t)}{X_{1}(t)Y_{1}(t)},
\]
and recalling the choice \eqref{eh:10} of $\gamma(t)$, we are led to the choice
\[
\phi(t)=-\frac{\eta'(t)}{n-1}+\frac{X_{1}^{2}(t)+Y_{1}^{2}(t)}{X_{1}(t)Y_{1}(t)}.
\]

For the above choice of $\phi(t)$ the condition \eqref{eh:11} is satisfied (it is just the hypothesis \eqref{eh:5}), thus concluding the proof of the theorem.

\end{proof}

\section{Corollaries and Remarks}\label{s:7}

Several remarks and corollaries are in order here.

\begin{remark}
In Theorem~\ref{t:all} we obtained a characterization of all Brownian co-adapted couplings for which the distance function $\rho$ is deterministic.  If the distance function $\rho$ is in addition a solution to an ordinary differential equation  of the form
\[
\rho'(t)=F(\rho(t))
\]
where $F:[0,\infty)\to\R$ is a $C^{1}$ function such that
\[
-(n-1)\sqrt{K}\tan(\sqrt{K}\rho/2)\le F(\rho)\le -(n-1)\sqrt{K}\tan(\sqrt{K}\rho/2)+\frac{2(n-1)\sqrt{K}}{\sin(\sqrt{K}\rho)},
\]
the construction given in the theorem also provides couplings $(X,Y)$ which are diffusions on the corresponding spaces.  The argument is simply the observation that the construction of the matrices $J(t),K(t)$ in this theorem depend only on $X(t)$ and $Y(t)$ and not on the time $t\ge0$. 
\end{remark}

Considering the case of the constant distance function $\rho$ in Theorem~\ref{t:10}, we obtain the following important result.

\begin{corollary}\label{c:1}
For any points $x,y\in \field{S}^{n-1}$ ($n\ge2$) there exists a co-adapted fixed-distance coupling, that is there exists a coupling $(X,Y)$ of Brownian motions on $\field{S}^{n-1}$ starting at $(x,y)$ for which $d(X(t),Y(t))=d(x,y)$ for all $t\ge0$.  Moreover, the joint processes $(X,Y)$ is a diffusion.
\end{corollary}

\begin{remark}\label{r:tv}  After writing this article, the second author learned from a private discussion with Thierry L\'evy an alternate short proof of the corollary above in the case $n=2$, which we briefly present below.   Consider $O(3)$, the set of $3\times 3$ orthogonal matrices, and for each $x\in \field{S}^{2}$ consider the map $\pi_{x}:O(3)\to \field{S}^{2}$ given by $\pi_{x}(A)=Ax$.  Consider the standard left-right invariant metric on $O(3)$, and the Riemannian structure associated to it.   Denoting by $\Delta^{O(3)}$ and $\Delta^{\field{S}^{2}}$ the Laplacian on $O(3)$, respectively on $\field{S}^{2}$, it can be shown that $\Delta^{O(3)}(f\circ \pi_{x})=\Delta^{\field{S}^{2}}(f)$, and as a consequence, if $Z(t)$ is a Brownian motion on $O(3)$, then $Z(t)x$ is a Brownian motion on $\field{S}^{2}$.  For arbitrary points $x,y\in\field{S}^{2}$, choosing $X(t)=Z(t)x$ and $Y(t)=Z(t)y$, provides a co-adapted fixed-distance coupling $(X(t),Y(t))$ of Brownian motions on $\field{S}^{2}$.
\end{remark}

Considering $\rho(t)=e^{-kt/2}\rho(0)$, with $k\le n-1$ in Theorem \ref{t:10} we obtain the following.

\begin{corollary}
For any points $x,y\in\field{S}^{n}$ ($n\ge1$) with $x\ne\pm y$ and $0\le k\le n-1$, there exists a co-adapted coupling $(X(t),Y(t))$ of Brownian motions  on $\field{S}^{n}$ such that $d(X(t),Y(t))=e^{-kt/2}d(x,y)$ for all $t\ge0$.

For $k<0$ and $n>1$, there also exists such a  coupling, but it only satisfies $d(X(t),Y(t))=e^{-kt/2}d(x,y)$, $0\le t\le t_{0}$, for some $t_{0}<\infty$.
\end{corollary}

Note that in the case  $k<0$ one cannot hope that the equality $d(X(t),Y(t))=e^{-kt/2}d(x,y)$ holds true for all times $t\ge 0$, since the exponential term is unbounded and the sphere $\field{S}^{n}$ is compact.

An interesting feature of the previous corollary is the fact that the maximum allowable value of $k$ is $n-1$.  It turns out that this has to do with the fact that the curvature of the sphere $\field{S}^n$ is 1 (more precisely the lower bound on the Ricci curvature is $n-1$, but the reasoning is not transparent from the extrinsic argument presented in the proof of the theorem.

\begin{remark}
Aside from the particular case when the function $\rho$ is constant, there are two other interesting particular cases of Theorem \ref{t:10}.  One is the extreme case in which the left inequality in the hypothesis \eqref{e:301} of the theorem is attained, namely
\begin{equation}\label{slow coupling}
\rho(t)=2\arcsin\left(e^{-(n-1)t/2}\sin(\rho(0)/2)\right), \qquad t\ge 0.
\end{equation}
The corresponding  coupling is a particular case of  shy coupling, in which the two processes do not couple, but approach each other exponentially fast.

The other case is the extreme case when the right side of the inequality \eqref{e:301} is attained, explicitly
\begin{equation}\label{fast coupling}
\rho(t)=2\arccos\left(e^{-(n-1)t/2}\cos(\rho(0)/2)\right), \qquad t\ge 0.
\end{equation}
The corresponding coupling is again a shy coupling (the processes do not couple in finite time), but it is a \emph{repulsive coupling}, in the sense that the distance between the two processes increases at an exponential rate to the maximum distance allowed on the sphere (the processes become antipodal in the limit).

Notice that the latter coupling is  related to the former coupling via the following simple observation.   If $(X,Y)$ is the coupling for which the distance function is given by \eqref{slow coupling}, then $(X(t),\tilde{Y}(t))$ with $\tilde{Y}(t)=-Y(t)$ is a coupling for which the distance function is given by \eqref{fast coupling}.
\end{remark}

\begin{remark}
For the case of the hyperbolic space $\field{H}^{n}$, we point out that for any function $\rho:[0,\infty)\to[0,\infty)$ for which there exists a co-adapted coupling of Brownian motions on $\field{H}^n$ with $d(X(t),Y(t))=\rho(t)$, from \eqref{eh:301} we get that
\[
2\mathrm{arcsinh}(e^{(n-1)t/2}\sinh(\rho(0)/2))\le\rho(t)\le 2\mathrm{arccosh}(e^{(n-1)t/2}\cosh(\rho(0)/2)),
\]
which in turn shows that
\[
\lim_{t\to\infty}\frac{\rho(t)}{t}=n-1.
\]
In particular, this shows that there cannot be an exponential growth at infinity for the distance function.

Moreover, any function $\rho(t)$ which can be realized as the distance between co-adapted Brownian motions must be increasing.  In particular, there are no fixed-distance couplings as in the case of the sphere or of the Euclidean space.

It is also interesting to point out that for $\rho(t)=e^{-kt}\rho(0)$, the inequality in \eqref{eh:301} is satisfied only for $k<0$ and for small $t$.  In turn, if we take $\rho(0)$ sufficiently small, we conclude that $k\le -(n-1)/2$.
\end{remark}

Though we studied the coupling of Brownian motions on the model spaces, we can extend part of the analysis to the case of constant curvature manifolds.  To state this, let $M$ be a manifold of constant curvature $K$ and assume that $i(M)$ is the injectivity radius of $M$.  If $\rho:[0,T)\to[0,i(M))$ is a continuous function such that \eqref{eg:1} then we can find a coupling of Brownian motions $(X,Y)$ such that $\rho(t)=d(X(t),Y(t))$.  The argument is simply based on the fact that $M$ is the quotient of $\field{M}_{K}^{n}$ by a discrete group of isometries.  Since the Brownian motions on $\field{M}_{K}^{n}$ stay within $i(M)$ distance, it turns out that their projections onto $M$ are Brownian motions staying within the cut locus of each other.


The initial motivation for writing the present article was to investigate the stochastic version of the Lion and Man problem presented in the introduction, and we conclude with an interpretation of our results in terms of it. Given a distance function $\rho$ satisfying the hypotheses of Theorem~\ref{t:all}, the Brownian Man can find a ``strategy'' to keep the Brownian Lion at distance $\rho(t)$ at time $t$.  In particular the Man can always find a strategy which keeps the Lion at fixed distance in the Euclidean space and on the sphere, which although it is trivial in the Euclidean case, it is completely non-trivial on the sphere. The Man can also find a strategy which increases his distance from the Lion in all model spaces.

While the Euclidean and the hyperbolic cases must be disappointing for the Lion (no distance-decreasing coupling), in the case of the sphere the Lion can find a strategy which brings him exponentially close in time to the Man, which should be sufficient for some practical purposes.

\textbf{Acknowledgements}.  We want to thank Wilfrid Kendall and Krzysztof Burdzy for several interesting discussions on the existence of fixed-distance coupling on the sphere which took place in the summer of 2009 while the first author visited University of Warwick.  This motivated us to undertake, extend and complete this program on manifolds, which will appear in a forthcoming paper.  The second author kindly thanks to Thierry L\'evy for his construction given in Remark \ref{r:tv}.

\providecommand{\href}[2]{#2}

\end{document}